\documentclass[11pt, a4paper, leqno]{amsart}

\usepackage[utf8]{inputenc}
\usepackage[T1]{fontenc}
\usepackage[english]{babel}
\usepackage{csquotes}

\usepackage[margin=1in]{geometry}
\usepackage{microtype}         
\usepackage{xcolor}            

\definecolor{color1}{HTML}{4169e1}
\definecolor{color2}{HTML}{d2d2d2} 

\usepackage{enumitem}          
\setlist[enumerate,1]{label=(\arabic*)} 

\usepackage{graphicx}          
\usepackage{booktabs}          

\usepackage{etoolbox}          
\usepackage{chngcntr}          

\usepackage{mathtools}         
\usepackage{amssymb}           
\usepackage{bm}                
\usepackage{mleftright}        
\mleftright

\usepackage{mathrsfs}          
\usepackage{upgreek}           

\usepackage{tikz}              
\usepackage{quiver}            

\newtheorem{theorem}{Theorem}
\newtheorem{lemma}[theorem]{Lemma}

\theoremstyle{definition}
\newtheorem{definition}[theorem]{Definition}

\newtheorem{remark}[theorem]{Remark} 

\theoremstyle{remark}

\counterwithout{equation}{section}

\mathtoolsset{showonlyrefs,showmanualtags}

\usepackage[colorlinks=true,
            linkcolor=color1,
            citecolor=color1,
            urlcolor=color2,
            pagebackref=true,
            unicode=true]{hyperref}

\usepackage[capitalise, nameinlink]{cleveref}

\crefname{theorem}{Theorem}{Theorems}
\crefname{lemma}{Lemma}{Lemmas}
\crefname{proposition}{Proposition}{Propositions}
\crefname{corollary}{Corollary}{Corollaries}
\crefname{conjecture}{Conjecture}{Conjectures}
\crefname{definition}{Definition}{Definitions}
\crefname{example}{Example}{Examples}
\crefname{problem}{Problem}{Problems}
\crefname{exercise}{Exercise}{Exercises}
\crefname{question}{Question}{Questions}
\crefname{remark}{Remark}{Remarks}
\crefname{notation}{Notation}{Notations}
\crefname{convention}{Convention}{Conventions}

\crefformat{equation}{(#2#1#3)}
\crefmultiformat{equation}{(#2#1#3)}{ and (#2#1#3)}{, (#2#1#3)}{, and (#2#1#3)}
\crefrangeformat{equation}{(#3#1#4)--(#5#2#6)}

\makeatletter
\newcommand\mkcal@one[1]{\expandafter\gdef\csname c#1\endcsname{\ensuremath{\mathcal{#1}}}}
\forcsvlist{\mkcal@one}{A,B,C,D,E,F,G,H,I,J,K,L,M,N,O,P,Q,R,S,T,U,V,W,X,Y,Z}

\newcommand\mksf@one[1]{\expandafter\gdef\csname s#1\endcsname{\ensuremath{\mathsf{#1}}}}
\forcsvlist{\mksf@one}{A,B,C,D,E,F,G,H,I,J,K,L,M,N,O,P,Q,R,S,T,U,V,W,X,Y,Z}

\newcommand\mkbb@one[1]{\expandafter\gdef\csname #1#1\endcsname{\ensuremath{\mathbb{#1}}}}
\forcsvlist{\mkbb@one}{A,B,C,D,E,F,G,H,I,J,K,L,M,N,O,P,Q,R,S,T,U,V,W,X,Y,Z}

\newcommand\mkbf@one[1]{\expandafter\gdef\csname b#1\endcsname{\ensuremath{\mathbf{#1}}}}
\forcsvlist{\mkbf@one}{A,B,C,D,E,F,G,H,I,J,K,L,M,N,O,P,Q,R,S,T,U,V,W,X,Y,Z}

\newcommand\mkrm@one[1]{\expandafter\gdef\csname r#1\endcsname{\ensuremath{\mathrm{#1}}}}
\forcsvlist{\mkrm@one}{A,B,C,D,E,F,G,H,I,J,K,L,M,N,O,P,Q,R,S,T,U,V,W,X,Y,Z}

\newcommand\mkfrak@one[1]{\expandafter\gdef\csname f#1\endcsname{\ensuremath{\mathfrak{#1}}}}
\forcsvlist{\mkfrak@one}{A,B,C,D,E,F,G,H,I,J,K,L,M,N,O,P,Q,R,S,T,U,V,W,X,Y,Z}
\makeatother

\DeclareMathOperator{\Db}{D^b}

\DeclareMathOperator{\Hom}{Hom}

\DeclareMathOperator{\id}{id}

\DeclareMathOperator{\Spec}{Spec}

\DeclareMathOperator{\Sym}{Sym}

\DeclareMathOperator{\Res}{Res}

\DeclareMathOperator{\Sgn}{sgn}

\title[A note on stability conditions on projective spaces]{A note on stability conditions on projective spaces}
\author{Yiran Cheng}
\address{Department of Mathematics, Imperial College London, London SW7 2AZ, United Kingdom}
\email{y.cheng@imperial.ac.uk}

\begin{document}

\begin{abstract}
    We give a new proof of Li's theorem on the existence of geometric Bridgeland stability conditions on the bounded derived category of coherent sheaves on projective spaces. These stability conditions can then be restricted to induce Bridgeland stability conditions on arbitrary smooth projective varieties.
\end{abstract}

\maketitle
\raggedbottom


Recently, Li \cite{Li26} established the existence of stability conditions on any smooth projective variety.
The key ingredient in Li's construction is the family of stability conditions on $\Db(\PP^n)$ satisfying Li's condition, so that they can be restricted to arbitrary smooth projective subvarieties. 
More precisely, he starts with a continuous family of $(\ZZ/2\ZZ)^n \rtimes \fS_n$-invariant stability conditions on $\Db(C^n)$ parametrized by $(a,b)\in \QQ_{>0} \times \QQ$, where $C$ is an elliptic curve and $(\ZZ/2\ZZ)^n \rtimes \fS_n$ acts on $C^n$ via factorwise involutions and permutations; see \cite[Theorem~5.2]{Li26}.
Then, a highly technical step in his proof is showing that such stability conditions descend to $\Db(\PP^n)$; see \cite[Theorem~6.3]{Li26}.
Moreover, the family satisfies Li's condition 
\begin{equation}\label{eq:Li}
    \phi^-_{a,b} (\cF \otimes \cO(mH)) < \phi^-_{a,b} (\cF[1])
\end{equation}
on $C^n$ \cite[Proposition~5.4]{Li26} and this property naturally descends to $\PP^n$ \cite[Corollary~6.4]{Li26}.
Finally, for any smooth subvariety, we can find a certain element in this family of stability conditions that restricts to this subvariety; see \cite[Theorem~6.5]{Li26}.

In this paper, we provide a slightly different approach. 
We view $\PP^{n-1}$ as a fiber of the Albanese map of $\Sym^n C$:
    \begin{equation}
        \begin{tikzcd}
            {\PP^{n-1}} & {\Sym^n C} \\
            {\Spec \CC} & C \, .
            \arrow[from=1-1, to=1-2]
            \arrow[from=1-1, to=2-1]
            \arrow["alb", from=1-2, to=2-2]
            \arrow[from=2-1, to=2-2]
        \end{tikzcd}
    \end{equation}
    According to \cite[Theorem~1.7]{cheng2025bridgeland}, any stability condition on $\Db(\Sym^n C)$ \emph{restricts} to its Albanese fibers $\Db(\PP^{n-1})$, since the target $C$ is abelian and the fibers are isotrivial. 
    Consider the natural morphism $\pi \colon [C^n/{\fS_n}] \to C^n / {\fS_n} \simeq \Sym^n C$ from the quotient stack to the quotient space.
Since this is a good quotient,
the pullback functor
\begin{equation}
    \pi^* \colon \Db(C^n/\fS_n) \to \Db([C^n/\fS_n])
\end{equation}
is fully faithful.
Viewing $\pi^* \Db(\Sym^n C \simeq C^n/\fS_n )$ as an admissible subcategory of the equivariant category $\Db([C^n /\fS_n])$ motivates the following main theorem of this paper.

\begin{theorem}\label{to quotient}
    Let $\sigma$ be any $\fS_n$-invariant stability condition on $\Db(C^n)$, which we also regard as a stability condition on the equivariant category $\Db([C^n /\fS_n])$.
    Then, the stability condition $\sigma$ \emph{restricts} to the admissible subcategory $\pi^*\Db(C^n/\fS_n)$.
\end{theorem}
For the precise definition of \emph{restriction}, see Definition~\ref{def:restriction}.
Here, $\fS_n$-invariant stability conditions on $\Db(C^n)$
induce stability conditions on the equivariant category $\Db([C^n/\fS_n])$ \cite{Pol07,MMS09}, and can be identified as stability conditions with a certain invariance on this equivariant category.

Therefore, if we start with the same family of stability conditions $\{\sigma_{a,b}\}$ on $\Db(C^n)$ and view it as a family of stability conditions on $\Db([C^n/\fS_n])$, by Theorem~\ref{to quotient} and \cite[Theorem~1.7]{cheng2025bridgeland}, we obtain the same family of geometric stability conditions on $\Db(\PP^{n-1})$ as in \cite{Li26}.
The same argument then shows certain stability conditions in this family can be restricted to subvarieties.

\begin{remark}
    The two constructions will produce the same family of stability conditions on $\Db(\PP^n)$.
    We also remark that in a separate paper \cite{CF26}, we establish a uniform bound on $(a,b)$ for which Li's condition \eqref{eq:Li} holds. 
    The same analysis applies here and yields the exact same bound for $a$.
\end{remark}

Throughout this paper, all varieties are defined over the field of complex numbers $\CC$, and all functors between derived categories are implicitly assumed to be derived. 
For a smooth projective variety $X$, we denote by $\Db(X)$ the bounded derived category of coherent sheaves on $X$. 
If a finite group $G$ acts on $X$, we use $\Db([X/G])$ to denote the bounded derived category of equivariant coherent sheaves. 
For the precise definition of stability conditions, we refer to \cite[Appendix~A]{BMS16}.



\section*{Proof of Theorem~\ref{to quotient}}
Before proceeding with the proof of the main theorem, we first make precise the notion of restricting a stability condition along a functor, a term heavily referenced in the introduction.
\begin{definition}\label{def:restriction}
    For a conservative triangulated functor $\Phi \colon \cD' \to \cD$ and a (pre-)stability condition $\sigma=(\cP,Z)$ on $\cD$, 
    we say the (pre-)stability condition \emph{restricts} to $\cD'$ if the naive pullback
    \begin{equation}
        \Phi^{-1}\sigma \coloneqq (\Phi^{-1} \cP , Z\circ \Phi) , \quad \text{where} \quad \Phi^{-1} \cP(\phi) \coloneqq \{ \cE \in \cD' \mid \Phi(\cE) \in \cP(\phi) \} ,
    \end{equation}
    defines a (pre-)stability condition on $\cD'$.
\end{definition}

Note that if a stability condition on $\cD$ restricts to a pre-stability condition on $\cD'$, it automatically defines a stability condition since it satisfies the support property with respect to the image lattice.

A special case is when $\cD' \subset \cD$ is a full subcategory. It is straightforward to observe the following.

\begin{lemma}
    To show that a stability condition restricts to a full subcategory, it suffices to show that the HN factors of any object in this subcategory remain in the subcategory.  
\end{lemma}

We now turn to the main theorem, beginning with the case $n=2$.

\begin{lemma}\label{lem:n=2}
    Theorem~\ref{to quotient} holds for $n=2$.
\end{lemma}

\begin{proof}
    Let $\pi \colon [C^2 / \fS_2] \to C^2 / \fS_2$ be the quotient and $\delta \colon [\Delta / \fS_2] \to [C^2 / \fS_2]$ be the embedding of the diagonal.
    We have a SOD
    \begin{equation} \label{eq:SOD}
        \Db([C^2 / \fS_2]) = \left\langle \pi^* \Db(C^2 / \fS_2), \delta_* \Db(\Delta) \right\rangle,
    \end{equation}
    see, for example, \cite[Theorem~1.6]{IU15} or \cite[Theorem~B]{PB19}.
    Here, we identify $\Db(\Delta)$ as the subcategory with trivial linearization in the splitting into two isotypic components
    \begin{equation}
        \Db([\Delta / \fS_2]) = \Db(\Delta) \oplus \bigl(\Db(\Delta) \otimes \Sgn \bigr),
    \end{equation}
    where $\Sgn$ denotes the nontrivial $\fS_2$-linearization.
    Note that any $\fS_2$-invariant stability condition on $\Db(C^2)$ can be viewed as a stability condition on $\Db([C^2 / \fS_2])$ that is $\otimes \Sgn$-invariant; see, for example, \cite[Theorem~4.8 and Lemma~4.11]{PPZ23}. 
    We prove the following claim, which immediately implies the lemma.
    
    \noindent \textbf{Claim:} Any $\otimes \Sgn$-invariant stability condition on $\Db([C^2 / \fS_2])$ restricts to both SOD components in \eqref{eq:SOD}.

    We first show that any stability condition on $\Db([C^2 / \fS_2])$ restricts to $\delta_* \Db(\Delta)$.
    Note that $\Delta$ is an isotrivial fiber of the map $(x_1, x_2) \mapsto x_1 -x_2$ to the elliptic curve $C$.
    By \cite[Theorem~1.7]{cheng2025bridgeland}, we know 
    that any stability condition on $\Db(C^2)$ restricts to $\Db(\Delta)$;
    in particular, any $\fS_2$-invariant stability condition on $\Db(C^2)$ restricts to (an $\fS_2$-invariant stability condition on) $\Db(\Delta)$.
    Consequently,
    any stability condition on $\Db([C^2 / \fS_2])$ restricts to $\delta_* \Db([\Delta / \fS_2])$, hence restricts further to the component $\delta_* \Db(\Delta)$ since there is no nonzero morphism between different isotypic components.
    
    Now we show that any $\otimes \Sgn$-invariant stability condition restricts to $\pi^* \Db(C^2 / \fS_2)$.
    We adopt the standard notation
    \begin{equation}
        \delta^* \dashv \delta_* \dashv \delta^! \qquad \text{and} \qquad \pi_! \dashv \pi^* 
    \end{equation}
    for the adjoint functors.
    By \cite[Theorem~2.1.2 and proof of Corollary~2.2.2]{Pol07}, it suffices to show that 
    \begin{equation}
        \phi^-(\pi^* \pi_! \cF) \geq \phi^-(\cF).
    \end{equation}
    By the triangle
    \begin{equation}\label{triangle}
        \delta_* \delta^! \cF \to  \cF \to \pi^* \pi_! \cF \to \delta_* \delta^! \cF[1] ,
    \end{equation}
    we know that
    \begin{equation}
        \phi^-( \pi^* \pi_! \cF ) \geq \min \left\{ \phi^-(\cF), \phi^-(\delta_* \delta^! \cF[1]) \right\},
    \end{equation}
    so it suffices to show
    \begin{equation} \label{ineq:n=2}
        \phi^-(\cF) \leq \phi^-(\delta_* \delta^! \cF [1]).
    \end{equation}
    Let $\cG$ be the minimal HN factor of $\delta_* \delta^! \cF$.
    From the first part of the proof, we know the stability condition restricts to $\delta_* \Db(\Delta)$, hence $\cG \simeq \delta_* \cG'$ for some $\cG'$. Since $\delta^!\delta_* \simeq \id$, we have
    \begin{equation}\label{eq:factor}
        \cG \simeq \delta_* \delta^!\cG.
    \end{equation}
    By Grothendieck--Verdier duality, we know that
    \begin{equation}\label{eq:duality}
        \delta^!(-) \simeq \delta^*(- \otimes \Sgn)[-1],
    \end{equation}
    where $\otimes \Sgn$ changes the linearization; see, for example, \cite[Proposition~1.20 and Theorem~3.8]{Nir09}.
    Therefore, we have
    \begin{equation}\thickmuskip=3mu \medmuskip=2mu
        \Hom(\delta_* \delta^! \cF,\cG) \simeq \Hom(\delta^!\cF ,\delta^! \cG) \overset{\eqref{eq:duality}}{\simeq} \Hom(\delta^*(\cF \otimes \Sgn [-1]) , \delta^!\cG) \overset{\eqref{eq:factor}}{\simeq} \Hom(\cF \otimes \Sgn, \cG [1]) .
    \end{equation}
    As $\Hom(\delta_* \delta^! \cF,\cG) \neq 0$, we know $\Hom(\cF \otimes \Sgn, \cG  [1]) \neq 0$, hence
    \begin{equation}
        \phi^-(\cF \otimes \Sgn) \leq \phi(\cG[1]) = \phi^-(\delta_* \delta^! \cF [1]) .
    \end{equation}
    Since the stability condition is $\otimes \Sgn$-invariant, this is equivalent to \eqref{ineq:n=2}.
\end{proof}

To approach the general case, we first fix some notation before the proof.
For any pair $i\neq j$, let $\Delta_{\{i,j\}} \coloneqq \{x_i=x_j\}$ be the big diagonal and let $\fS_{\{i,j\}}\subset \fS_n$ be the subgroup that switches $x_i$ with $x_j$, which therefore fixes $\Delta_{\{i,j\}}$.
We denote the relevant morphisms as in the following commutative diagram
\begin{equation}\label{eq:diagram}
    \begin{tikzcd}
    	{[\Delta_{\{i,j\}} / \fS_{\{i,j\}}]} & {[C^n / \fS_{\{i,j\}}]} & {[C^n / \fS_n]} \\
    	& {C^n / \fS_{\{i,j\}}} & {C^n / \fS_n} \, ,
    	\arrow["{{\delta_{\{i,j\}}}}", from=1-1, to=1-2]
    	\arrow["{{\varpi_{\{i,j\}}}}", from=1-2, to=1-3]
    	\arrow["{{\pi_{\{i,j\}}}}", from=1-2, to=2-2]
    	\arrow["\pi", from=1-3, to=2-3]
    	\arrow[from=2-2, to=2-3]
    \end{tikzcd}
\end{equation}
where the pullback along the projection $\varpi_{\{i,j\}}$ is nothing but the restriction functor
\begin{equation}
    \varpi_{\{i,j\}}^* = \Res^{\fS_n}_{\fS_{\{i,j\}}} \colon  \Db([C^n / \fS_n]) \to \Db([C^n / \fS_{\{i,j\}}]) .
\end{equation}
We have the following straightforward observation.
\begin{lemma}\label{intersection}
    Inside the equivariant category $\Db([C^n / \fS_n])$, we have the equality
    \begin{equation}
        \pi^* \Db(C^n / \fS_n) = \bigcap_{1 \leq i<j \leq n} (\varpi_{\{i,j\}}^*)^{-1} \pi_{\{i,j\}}^*  \Db(C^n / \fS_{\{i,j\}}).
    \end{equation}
\end{lemma}
\begin{proof}
    By the derived Kempf's criterion (see, for example, \cite[Theorem~1.3]{Nev08}),
    an object $\cE \in \Db([C^n / \fS_n])$ lies in $\pi^* \Db(C^n / \fS_n)$ if and only if for every closed point $x$, the derived fiber $\imath_x^* \cE$ carries the trivial $(\fS_n)_x$-linearization, where $(\fS_n)_x$ is the stabilizer subgroup at $x$.
    Similarly, an object $\cE \in \Db([C^n / \fS_n])$ lies in $ (\varpi_{\{i,j\}}^*)^{-1} \pi_{\{i,j\}}^*  \Db(C^n / \fS_{\{i,j\}})$ if and only if for every closed point $x$, the derived fiber $\imath_x^* \cE$ carries the trivial $(\fS_{\{i,j\}})_x$-linearization.
    Since $\imath_x^* \cE$ is a direct sum of shifted vector spaces, we are really looking at linear representations of $(\fS_n)_x$ and $(\fS_{\{i,j\}})_x$ on every cohomology $\cH^k(\imath_x^* \cE)$.

    It is clear that the left-hand side is contained in the right-hand side, since a trivial $(\fS_n)_x$-representation restricts to a trivial $(\fS_{\{i,j\}})_x$-representation for each $\{i,j\}$.
    For the other direction,
    we claim that for any closed point $x \in C^n$, all subgroups $(\fS_{\{i,j\}})_x$ always generate $(\fS_n)_x$.
    Indeed, up to an $\fS_n$-action, we may assume $x=(x_1, \dots, x_n)$ with
    \begin{equation}
        x_{1}= \dots = x_{\lambda_1} , \qquad x_{\lambda_1+1}= \dots = x_{\lambda_1+\lambda_2} , \quad \dots , \quad x_{n-\lambda_\ell +1}= \dots = x_{n} .
    \end{equation}
    Then the group $(\fS_n)_x$ is precisely of the form
    \begin{equation}
        (\fS_n)_x \simeq \fS_{\{ 1, \dots , \lambda_1 \}} \times \fS_{\{ \lambda_1 +1, \dots , \lambda_1 + \lambda_2 \}} \times \dots \times \fS_{\{ n -\lambda_\ell +1 , \dots , n \}} 
    \end{equation}
    (which is called a Young subgroup).
    In particular, it is generated by transpositions, which are precisely induced by these $(\fS_{\{i,j\}})_x$.
\end{proof}

For any pair $i\neq j$, we have a SOD
\begin{equation}\label{eq:SOD'}
    \Db\left([C^n / \fS_{\{i,j\}}]\right) = \left\langle \pi_{\{i,j\}}^* \Db(C^n / \fS_{\{i,j\}}), \delta_{\{i,j\} *}\Db(\Delta_{\{i,j\}}) \right\rangle
\end{equation}
from \eqref{eq:SOD}. Here, we identify $\Db(\Delta_{\{i,j\}})$ as the subcategory with trivial linearization in the splitting 
\begin{equation}
    \Db([\Delta_{\{i,j\}} / \fS_{\{i,j\}}]) = \Db(\Delta_{\{i,j\}}) \oplus \bigl( \Db(\Delta_{\{i,j\}}) \otimes \Sgn_{\{i,j\}} \bigr),
\end{equation}
where $\Sgn_{\{i,j\}}$ denotes the nontrivial $\fS_{\{i,j\}}$-linearization.
The proof of Lemma~\ref{lem:n=2} implies the following.
\begin{lemma}\label{lem:quotient}
    Any $\otimes \Sgn_{\{i,j\}}$-invariant stability condition on $\Db([C^n / \fS_{\{i,j\}} ])$ restricts to both SOD components in \eqref{eq:SOD'}.
\end{lemma}
\begin{proof}
    This follows literally from the same proof of Lemma~\ref{lem:n=2}, with
    \begin{equation}
        C^2, \qquad \fS_2, \qquad \Delta, \qquad \pi, \qquad \delta, \qquad \Sgn
    \end{equation}
    replaced by
    \[
    C^n, \quad \fS_{\{i,j\}}, \quad \Delta_{\{i,j\}}, \quad \pi_{\{i,j\}}, \quad \delta_{\{i,j\}}, \quad \Sgn_{\{i,j\}}. 
    \]
    We walk through the proof below for the reader's convenience.

    We first show that any stability condition on $\Db\left([C^n / \fS_{\{i,j\}}]\right)$ restricts to $\delta_{\{i,j\} *}\Db(\Delta_{\{i,j\}})$.
    The big diagonal $\Delta_{\{i,j\}}$ is an isotrivial fiber of the map $(x_1, \dots, x_n) \mapsto x_i -x_j$ to the elliptic curve $C$.
    By \cite[Theorem~1.7]{cheng2025bridgeland}, we know 
    that any stability condition on $\Db(C^n)$ restricts to $\Db(\Delta_{\{i,j\}})$.
    Consequently,
    any stability condition on $\Db([C^n / \fS_{\{i,j\}}])$ restricts to $\delta_* \Db([\Delta_{\{i,j\}} / \fS_{\{i,j\}}])$, hence restricts further to $\delta_{\{i,j\}*} \Db(\Delta_{\{i,j\}})$.
    
    Now we show that any $\otimes \Sgn_{\{i,j\}}$-invariant stability condition restricts to $\pi_{\{i,j\}}^* \Db(C^n / \fS_{\{i,j\}})$.
    By \cite[Theorem~2.1.2 and proof of Corollary~2.2.2]{Pol07}, it suffices to show that 
    \begin{equation}
        \phi^-(\pi_{\{i,j\}}^* \pi_{\{i,j\}!} \cF) \geq \phi^-(\cF).
    \end{equation}
    By the distinguished triangle associated with the SOD \eqref{eq:SOD'},
    we know that
    \begin{equation}
        \phi^-( \pi_{\{i,j\}}^* \pi_{\{i,j\}!} \cF ) \geq \min \left\{ \phi^-(\cF), \phi^-(\delta_{\{i,j\}*} \delta_{\{i,j\}}^! \cF[1]) \right\},
    \end{equation}
    so it suffices to show
    \begin{equation} \label{ineq:n=2'}
        \phi^-(\cF) \leq \phi^-(\delta_{\{i,j\}*} \delta_{\{i,j\}}^! \cF [1]).
    \end{equation}
    Let $\cG$ be the minimal HN factor of $\delta_{\{i,j\}*} \delta_{\{i,j\}}^! \cF$.
    From the first part of the proof, we know the stability condition restricts to $\delta_{\{i,j\}*} \Db(\Delta_{\{i,j\}})$; together with $\delta_{\{i,j\}}^! \delta_{\{i,j\}*}  \simeq \id$, we have
    \begin{equation}\label{eq:factor'}
        \cG \simeq \delta_{\{i,j\}*} \delta_{\{i,j\}}^!\cG.
    \end{equation}
    By Grothendieck--Verdier duality, we know that
    \begin{equation}\label{eq:duality'}
        \delta_{\{i,j\}}^!(-) \simeq \delta_{\{i,j\}}^*(- \otimes \Sgn_{\{i,j\}})[-1].
    \end{equation}
    Therefore, we have
    \begin{align}
    \Hom(\delta_{\{i,j\}*} \delta_{\{i,j\}}^! \cF,\cG) & \simeq \Hom(\delta_{\{i,j\}}^!\cF ,\delta_{\{i,j\}}^! \cG) \\
    & \overset{\mathclap{\eqref{eq:duality'}}}{\simeq} \Hom(\delta_{\{i,j\}}^*(\cF \otimes \Sgn_{\{i,j\}} [-1]) , \delta_{\{i,j\}}^!\cG) \overset{\eqref{eq:factor'}}{\simeq} \Hom(\cF \otimes \Sgn_{\{i,j\}}, \cG [1]) .
    \end{align}
    As $\Hom(\delta_{\{i,j\}*} \delta_{\{i,j\}}^! \cF,\cG) \neq 0$, we know $\Hom(\cF \otimes \Sgn_{\{i,j\}}, \cG  [1]) \neq 0$, hence
    \begin{equation}
        \phi^-(\cF \otimes \Sgn_{\{i,j\}}) \leq \phi(\cG[1]) = \phi^-(\delta_{\{i,j\}*} \delta_{\{i,j\}}^! \cF [1]) .
    \end{equation}
    Since the stability condition is $\otimes \Sgn_{\{i,j\}}$-invariant, this is equivalent to \eqref{ineq:n=2'}.
\end{proof}

\begin{proof}[Proof of Theorem~\ref{to quotient}]
    It suffices to show that for any object $\cF \in \pi^* \Db(C^n / \fS_n)$, its HN factors also belong to $\pi^* \Db(C^n / \fS_n)$.
    From the commutativity of diagram \eqref{eq:diagram}, we have $\varpi_{\{i,j\}}^* \cF \in \pi_{\{i,j\}}^*\Db(C^n / \fS_{\{i,j\}})$ for every pair $i \neq j$.
    Lemma~\ref{lem:quotient} then implies that the HN factors of $\varpi_{\{i,j\}}^* \cF$ remain inside the subcategory $\pi_{\{i,j\}}^*\Db(C^n / \fS_{\{i,j\}})$.
    Since this condition holds for all pairs $i \neq j$, applying Lemma~\ref{intersection} allows us to conclude that the HN factors of the original object $\cF$ are contained in $\pi^* \Db(C^n / \fS_n)$, as desired.
\end{proof}

\bigskip

\subsection*{Acknowledgements}
I would like to thank Soheyla Feyzbakhsh, Lie Fu, Chunyi Li, Dongjian Wu, and Xiaolei Zhao for helpful discussions. 
The author was supported by the Royal Society through the University Research Fellowship URF/R1/231191.

\bigskip

\bigskip

\bibliographystyle{alphaurl}
\bibliography{ref}

\end{document}